\documentclass[12pt]{article}
\usepackage{epsfig}
\usepackage{amsfonts}
\usepackage{amsthm}
\usepackage{amsmath}
\usepackage{amssymb}

\allowdisplaybreaks[1]

\newtheorem{theorem}{Theorem}
\newtheorem{lemma}[theorem]{Lemma}
\newtheorem{corollary}[theorem]{Corollary}
\newtheorem{claim}{Claim}
\def\const{{0.4352}}
\def\constfrac{{2.2978}}
\def\constcut{{0.8704}}
\def\constcutnot{{0.8834}}
\def\uppercut{{0.9351}}
\def\upper{{0.455}} 
\def\hoppen{{0.4328}}
\def\duckworth{{0.4347}} 
\def\experiment{{0.439}}
\def\experimentnew{{0.447}}
\def\rounds{{307449}}
\def\lra{{\leftrightarrow}}
\def\C{{\cal C}}
\def\D{{\cal D}}

\def\J{{\cal J}}
\def\T{{\cal T}}

\def\cond{{ \; | \; }}
\def\Pr{{\mathbf P}}
\def\dint{{\mathbf d}}
\def\vec{\overrightarrow}
\def\bar{\overline}

\usepackage{listings}
\usepackage{textcomp}
\usepackage{setspace}
\lstnewenvironment{python}[1][]{
\lstset{
language=python,
basicstyle=\ttfamily\scriptsize\setstretch{1},
showstringspaces=false,
alsoletter={1234567890},
otherkeywords={\ , \}, \{},
emph={access,and,break,class,continue,def,del,elif,else,%
except,exec,finally,for,from,global,if,import,in,is,%
lambda,not,or,pass,print,raise,return,try,while},
emph={[2]True, False, None, self},
emphstyle=[2],
emph={[3]from, import, as},
emphstyle=[3],
upquote=true,
morecomment=[s]{"""}{"""},
commentstyle=\slshape,
emph={[4]1, 2, 3, 4, 5, 6, 7, 8, 9, 0},
emphstyle=[4],
literate=*{:}{{:}}{1}%
	{=}{{=}}{1}%
	{-}{{-}}{1}%
	{+}{{+}}{1}%
	{*}{{*}}{1}%
	{!}{{!}}{1}%
	{(}{{(}}{1}%
	{)}{{)}}{1}%
	{[}{{[}}{1}%
	{]}{{]}}{1}%
	{<}{{<}}{1}%
	{>}{{>}}{1},%
#1
}}{}

\begin{document}
\title{Fractional colorings of cubic graphs with large girth}
\author{Franti\v sek Kardo\v s\thanks{Institute of Mathematics, Faculty of Science, University of Pavol Jozef \v Saf\'arik, Jesenn\'a 5, 041 54 Ko\v sice, Slovakia. E-mail: {\tt frantisek.kardos@upjs.sk}. This author was supported by Slovak Research and Development Agency under the contract no. APVV-0007-07.}
\and	Daniel Kr\'al'\thanks{Department of Applied Mathematics and Institute for Theoretical Computer Science (ITI), Faculty of Mathematics and Physics, Charles University, Malostransk\'e n\'am\v est\'\i{} 25, 118 00 Prague 1, Czech Republic. E-mail: {\tt kral@kam.mff.cuni.cz}. Institute for Theoretical computer science is supported as project 1M0545 by Czech Ministry of Education. This author was supported by the grant GACR 201/09/0197 and in part
was also supported by the grant GAUK~60310.
}
\and	Jan Volec\thanks
{Department of Applied Mathematics, Faculty of Mathematics and Physics, Charles University, Malostransk\'e n\'am\v est\'\i{} 25, 118 00 Prague 1, Czech Republic. E-mail: {\tt volec@kam.mff.cuni.cz}. This author was supported by the grant GACR 201/09/0197.}
}

\date{}
\maketitle
\begin{abstract}
We show that every (sub)cubic $n$-vertex graph with sufficiently large girth
has fractional chromatic number at most $\constfrac$ which implies that
it contains an independent set of size at least $\const n$. Our bound on the independence number
is valid to random cubic graphs as well as it improves existing lower bounds
on the maximum cut in cubic graphs with large girth.
\end{abstract}

\section{Introduction}
\label{sect-intro}

An {\it independent set} is a subset of vertices such that no two of them
are adjacent. The {\it independence number} $\alpha(G)$ of a graph $G$
is the size of the largest independent set in $G$. In this paper,
we study independent sets in cubic graphs with large girth.
Recall that a graph $G$ is {\em cubic} if every vertex of $G$
has degree three, $G$ is {\em subcubic} if every vertex has degree
at most three, and
the {\em girth} of $G$ is the length of the shortest cycle of $G$.

A {\em fractional coloring} of a graph $G$ is an assignment of weights
to independent sets in $G$ such that for each vertex $v$ of $G$,
the sum of weights of the sets containg $v$ is at least one.
The {\em fractional chromatic number $c_f(G)$} of $G$
is the minimum sum of weights of independent sets forming
a fractional coloring. The main result of this paper asserts that
if $G$ is a cubic graph with sufficiently large girth, then
there exists a probability distribution on its independent sets
such that each vertex is contained in an independent set drawn
based on this distribution with probability at least $\const$.
This implies that every $n$-vertex cubic graph $G$ with large girth
contains an independent set with at least $\const n$ and its fractional
chromatic number is at most $\constfrac$ (to see the latter,
consider the probability distribution and assign each independent
set $I$ in $G$ a weight equal to $c p(I)$ where $p(I)$ is the probability
of $I$ and $c=1/\const$). In addition, our lower
bound on the independence number also translate to random cubic
graphs.

Let us now survey previous results in this area.
Inspired by Ne\v set\v ril's Pentagon Conjecture~\cite{bib-nesetril},
Hatami and Zhu~\cite{bib-hatami} showed that every cubic graph
with sufficiently large girth has fractional chromatic number
at most $8/3$. For the independence number, Hoppen~\cite{bib-hoppen}
showed that every $n$-vertex cubic graph with sufficiently large girth
has independence number at least $\hoppen n$; this bound matches
an earlier bound of Wormald~\cite{bib-wormald}, also independently proven 
by Frieze and Suen~\cite{bib-frieze}, for random cubic graphs.
The bound for random cubic graphs was further improved by Duckworth and
Zito~\cite{bib-duckworth} to $\duckworth n$, providing an improvement
of the bound unchallenged for almost 15 years.

\subsection{Random cubic graphs}

The independence numbers of random cubic graphs and cubic graphs with large girth
are closely related. Here, we consider the model where a random cubic graph
is randomly uniformly chosen from all cubic graphs on $n$ vertices.

Any lower bound on the independence number for the class of cubic graphs 
with large girth is also a lower bound for random cubic graphs.
First observe that a lower bound on the independence number for cubic graphs
with large girth translates to subcubic graphs with large girth:
assume that $H$ is a subcubic graph with $m_1$ vertices of degree one and
$m_2$ vertices of degree two. Now consider an $(2m_1+m_2)$-regular graph
with high girth and replace each vertex of it with copy of $H$ in such a way
that its edges are incident with vertices of degree one and two in the copies.
The obtained graph $G$ is cubic and has large girth; an application
of the lower bound on the independence number for cubic graphs yields that
one of the copies of $H$ contains a large independent set, too.

Since a random cubic graph contains asymptotically almost surely only
$o(n)$ cycles shorter than a fixed integer $g$~\cite{bib-wormald-shortcycles},
any lower bound for cubic graphs with large girth also applies to random
cubic graphs.
Conversely, Hoppen and Wormald~\cite{bib-wormald-private} have recently developed
a technique for translating lower bounds (obtained in a specific but quite general way)
for the independence number of a random cubic graph to cubic graphs with large girth.

In the other direction, any upper bound on the independence number for a~random cubic graph also applies
to cubic graphs with large girth (just remove few short cycles from a random graph).
The currently best upper bound for random cubic graphs with $n$ vertices and
thus for cubic graphs with large girth is $\upper n$ derived by McKay~\cite{bib-mckay}. 
In~\cite{bib-mckay}, McKay mentions that experimental evidence suggests that
almost all $n$-vertex cubic graphs contain independent sets of size $\experiment n$;
newer experiments of McKay~\cite{bib-mckay-private} then suggests that
the lower bound can even be $\experimentnew n$.

\subsection{Maximum cut}

Our results also improve known bounds on the size of the maximum cut in cubic graphs with large girth.
A {\em cut} of a graph $G$ is the set of edges such that the vertices of $G$ can be partitioned
into two subsets $A$ and $B$ and the edges of the cut have one end-vertex in $A$ and the other in $B$.
The {\em maximum cut} is a cut with the largest number of edges.
It is known that if $G$ is an $m$-edge cubic graph with sufficiently large girth,
then its maximum cut has size at least $6m/7-o(m)$~\cite{bib-zyka}. On the other hand,
there exist $m$-edge cubic graphs with large girth with the maximum cut of size at most $\uppercut$.
This result was first announced by McKay~\cite{bib-mckay-cut}, the proof can be found in~\cite{bib-hh}.

If $G$ has an independent set of size $\alpha$, then it also has a cut of size $3\alpha$ (put
the independent set on one side and all the other vertices on the other side). So, if $G$ is a cubic
graph with $n$ vertices and $m=3n/2$ edges and the girth of $G$ is sufficiently large,
then it has a cut of size at least $3\cdot\const n=\constcut m$. Let us remark that
this bound can be further improved to $\constcutnot m$ by considering a randomized
procedure tuned up for producing a cut of large size. Since we focus
on independent sets in cubic graphs with large girth in this paper,
we omit details how the improved bound can be obtained.

\section{Structure of the proof}

Our proof is inspired by the proof of Hoppen from~\cite{bib-hoppen}.
We develop a procedure for obtaining a random independent set in a cubic
graph with large girth. Our main result is thus the following.

\begin{theorem}
\label{thm-distribution}
There exists $g>0$ such that for every cubic graph $G$ with girth at least $g$,
there is a probability distribution such that each vertex is contained in an independent
set drawn according to this distribution with probability at least $\const$.
\end{theorem}

We describe our procedure for obtaining a random independent set in a cubic
graph in Section~\ref{sect-procedure}. To analyze its performance, we first
focus on its behavior for infinite cubic trees in Section~\ref{sect-tree}.
The core of our analysis is the independence lemma (Lemma~\ref{indep})
which is used to simplify the recurrence relation appearing in the analysis.
We then show that the procedure can be modified for cubic graphs
with sufficiently large girth while keeping its performance.
The actual performance of our randomized procedure is based
on solving the derived recurrences numerically.

As we have explained in Section~\ref{sect-intro},
Theorem~\ref{thm-distribution} has the following three corollaries.

\begin{corollary}
\label{cor-indep}
There exists $g>0$ such that every $n$-vertex subcubic graph with girth at least $g$ 
contains an independent set of size at least $\const n$.
\end{corollary}

\begin{corollary}
\label{cor-frac}
There exists $g>0$ such that every cubic graph with girth at least $g$
has the fractional chromatic number at most $\constfrac$.
\end{corollary}

\begin{corollary}
\label{cor-cut}
There exists $g>0$ such that every cubic graph $G$ with girth at least $g$
has a cut containing at least $\constcut m$ edges where $m$ is the total
number of edges of $G$.
\end{corollary}

\section{Description of the randomized procedure}
\label{sect-procedure}

The procedure for creating a random independent set
is parametrized by three numbers:
the number $K$ of its rounds and probabilites $p_1$ and $p_2$.

Throughout the procedure,
vertices of the input graph have one of the three colors: white, blue and red.
At the beginning, all vertices are white.
In each round, some of the white vertices are recolored in such a way that
{\em red} vertices always form an independent set and
all vertices adjacent to red vertices as well as some of other vertices (see details below)
are {\em blue}. All other vertices of the input graph are {\em white}.
Red and blue vertices are never again recolored during the procedure.
Because of this,
when we talk of a {\em degree} of a vertex, we mean the number of its white neighbors.
Observe that the neighbors of a white vertex that are not white must be blue.

The first round of the procedure is special and differ from the other rounds.
As we have already said, at the very beginning, all vertices of the input graph
are white. In the first round, we randomly and independently with probability $p_1$
mark some vertices as active. Active vertices with no active neighbor become red and
vertices with at least one active neighbor become blue. In particular, if two adjacent
vertices are active, they both become blue (as well as their neighbors).
At this point, we should note that the probabilty $p_1$ will be very small.

In the second and the remaining rounds,
we first color all white vertices of degree zero by red.
We then consider all paths formed by white vertices
with end vertices of degree one or three and with all inner vertices of degree two.
Based on the degrees of their end vertices, we refer to these paths as paths of type
$1\lra1$, $1\lra3$ or $3\lra3$.
Note that these paths can have no inner vertices.
Each vertex of degree two is now activated with probability $p_2$ independently of
all other vertices.

For each path of type $1\lra3$, we color the end vertex of degree one with red and
we then color all the inner vertices with red and blue in the alternating way.
If the neighbor of the end vertex of degree three becomes red,
we also color the end vertex of degree three with blue.
In other words, we color the end vertex of degree three by blue
if the path has an odd length.
For a path of type $1\lra1$, we choose randomly one of its end vertices,
color it red and color all the remaining vertices of the path with red and blue
in the alternating way.
We refer to the choosen end vertex as the beginning of this path.
Note that the facts whether and which vertices of degree two on the paths of type $1\lra1$ or $1\lra3$
are activated do not influence the above procedure.

Paths of type $3\lra3$ are processed as follows. A path of type $3\lra3$
becomes active if at least one of its inner vertices is activated, i.e., a path of length $\ell$
is active with probability $1-(1-p_2)^{\ell-1}$. Note that paths with no inner vertices,
i.e., edges between two vertices of degree three, are never active.
For each active path, flip the fair coin to select one of its end vertices of degree three,
color this vertex by blue and its neighbor on the path by red.
The remaining inner vertices of the path are colored with red and blue in the alternating way.
Again we refer to the choosen end vertex as the beginning of this path.
The other end vertex of degree three is colored blue if its neighbor on the path becomes red,
i.e., if the path has an even length.

Note that a vertex that has degree three at the beginning of the round cannot become red during the round
but it can become blue because of several different paths ending at it.

\section{Analysis for infinite cubic trees}
\label{sect-tree}

Let us start with introducing some notation used in the analysis.
For any edge $uv$ of a cubic graph $G$, $T_{u,v}$ is the component of $G-v$ containing
the vertex $u$; we sometimes refer to $u$ as to the root of $T_{u,v}$.
 If $G$ is the infinite cubic tree, then it is a union of
$T_{u,v}, T_{v,u}$ and the edge $uv$. The subgraph of $T_{u,v}$ induced by vertices
at distance at most $d$ from $u$ is denoted by $T_{u,v}^d$.
Observe that if the girth of $G$ is larger than $2d+1$,
all the subgraphs $T_{u,v}^d$ are isomorphic to the same rooted tree $\T^d$ of depth $d$.
The infinite rooted tree with the root of degree two and all inner vertices of degree three will be denoted as $T^\infty$.

Since any automorphism of the cubic tree yields an automorphism of the probability space of
the vertex colorings constructed by our procedure,
the probability that a vertex $u$ has a fixed color after $k$ rounds
does not depend on the choice of $u$.
Hence, we can use $w_k$, $b_k$ and $r_k$ to denote a probability that a fixed
vertex of the infinite cubic tree is white, blue and red, respectively, after $k$ rounds.
Similarly, $w^i_k$ denotes the probability that a fixed vertex has degree $i$ after $k$ rounds
conditioned by the event that it is white after $k$ rounds, i.e., the probability that
a fixed vertex white and has degree two after $k$ rounds is $w_k\cdot w^2_k$.

The sets of white, blue and red vertices after $k$ rounds of the randomized procedure
described in Section~\ref{sect-procedure} will be denoted by $W_k$, $B_k$ and $R_k$,
respectively. Similarly, $W^i_k$ denotes the set of white vertices with degree $i$
after $k$ rounds, i.e., $W_k=W^0_k \cup W^1_k\cup W^2_k\cup W^3_k$.
Finally, $c_k\left(T_{u,v}\right)$ denotes the coloring of $T_{u,v}$ and
$c_k\left(T_{u,v}^d\right)$ the coloring of $T_{u,v}^d$ after $k$ rounds.
The set $\C^d_k$ will consist of all possible colorings $\gamma$ of~$\T^d$
such that the probability of $c_k\left(T_{u,v}^d\right)=\gamma$ is non-zero
in the infinite cubic tree (note that this probability does not depend on the edge $uv$) and
such that the root of $\T^d$ is colored white.
We extend this notation and use $\C^\infty_k$ to denote all such colorings of~$\T^\infty$. 

Since the infinite cubic tree is strongly edge-transitive, we conclude
that the probability that $T_{u,v}$ has a given coloring from $\C^\infty_k$ after $k$ rounds
does not depend on the choice of $uv$. Similarly, the probability that both $u$ and $v$
are white after $k$ rounds does not depend on the choice of the edge $uv$. To simplify
our notation, this event is denoted by $uv\subseteq W_k$.
Finally, the probability that $u$ has degree $i\in\{1,2,3\}$ after $k$ rounds conditioned by $uv\subseteq W_k$
does also not depend on the choice of the edge $uv$. This probability is denoted by $q^i_k$.

\subsection{Independence lemma}

We now show that the probability that both the vertices of an edge $uv$ are white after $k$ rounds
can be computed as a product of two probabilities. This will be crucial in the proof of the Independence Lemma.
To be able to state the next lemma, we need to introduce additional notation. The probability
$P_k(u,v,\gamma)$ for $\gamma\in\C^\infty_{k-1}$ is the probability that
$$\Pr\bigl[u \; \mbox{\rm stays white regarding} \; T_{u,v} \cond c_{k-1}\left(T_{u,v}\right)=\gamma \land v \in W_{k-1}\bigr]$$
where the phrase ``$u$ stays white regarding $T_{u,v}$'' means
\begin{itemize}
\item if $k=1$, that neither $u$ nor a neighbor of it in $T_{u,v}$ is active, and
\item if $k>1$, that neither
  \begin{itemize}
  \item $u$ has degree one after $k-1$ rounds,
  \item $u$ has degree two after $k-1$ rounds and the path formed by vertices of degree two in $T_{u,v}$ ends at a vertex of degree one,
  \item $u$ has degree two after $k-1$ rounds, the path formed by vertices of degree two in $T_{u,v}$ ends at a vertex of degree three and at least one of the vertices of degree two on this path is activated, nor
  \item $u$ has degree three after $k-1$ rounds and is colored blue because of a path of type $1\lra3$ or $3\lra 3$
        fully contained in $T_{u,v}$.
  \end{itemize}
\end{itemize}
Informally, this phrase represents that there is no reason for $u$ not to stay white
based on the coloring of $T_{u,v}$ and vertices in $T_{u,v}$ activated in the $k$-th round.

\begin{lemma}
\label{claim}
Consider the randomized procedure for the infinite cubic tree.
Let $k$ be an integer, $uv$ an edge of the tree and
$\gamma_u$ and $\gamma_v$ two colorings from $\C^\infty_{k-1}$.
The probability
\begin{equation}
\label{eq-claim}
\Pr\bigl[uv \subseteq W_k \cond c_{k-1}\left(T_{u,v}\right)=\gamma_u \land c_{k-1}\left(T_{v,u}\right)=\gamma_v\bigr]\;\mbox{,}
\end{equation}
i.e., the probability that both $u$ and $v$ are white after $k$ rounds conditioned by $c_{k-1}\left(T_{u,v}\right)=\gamma_u$ and
$c_{k-1}\left(T_{v,u}\right)=\gamma_v$, is equal to
$$P_{k}(u,v,\gamma_u)\cdot P_{k}(v,u,\gamma_v)\;\mbox{.}$$
\end{lemma}

\begin{proof}
We distinguish the cases $k=1$ and $k>1$.
If $k=1$, $\C^\infty_0$ contains a single coloring $\gamma_0$ where all vertices are white.
Hence, the probability (\ref{eq-claim}) is
$\Pr\bigl[uv \subseteq W_1 \cond c_{k-1}\left(T_{u,v}\right)=\gamma_0 \land c_{k-1}\left(T_{v,u}\right)=\gamma_0\bigr]$ and
it is equal to $\Pr\bigl[uv\subseteq W_1\bigr]=(1-p_1)^6$.
On the other hand, $P_1(u,v,\gamma_0)=P_1(v,u,\gamma_0)=(1-p_1)^3$.
The assertion of the lemma follows.

Suppose that $k>1$. Note that the colorings $\gamma_u$ and $\gamma_v$
completely determine the coloring after $k-1$ rounds.
If $u$ has degree one after $k-1$ rounds,
then (\ref{eq-claim}) is zero as well as $P_k(u,v,\gamma_u)=0$.
If $u$ has degree two and lie on a path of type $1\lra1$ or $1\lra3$,
then (\ref{eq-claim}) is zero and $P_k(u,v,\gamma_u)=0$ or $P_k(v,u,\gamma_v)=0$
depending which of the trees $T_{u,v}$ and $T_{v,u}$ contains the vertex of degree one.
If $u$ has degree two and lie on a path of type $3\lra3$ of length $\ell$,
then (\ref{eq-claim}) is equal to $(1-p_2)^{\ell-1}$.
Let $\ell_1$ and $\ell_2$ be the number of vertices of degree two on this path
in $T_{u,v}$ and $T_{v,u}$, respectively. Observe that $\ell_1+\ell_2=\ell-1$.
Since $P_k(u,v,\gamma_u)=(1-p_2)^{\ell_1}$ and $P_k(v,u,\gamma_v)=(1-p_2)^{\ell_2}$,
the claimed equality holds.

Hence, we can now assume that the degree of $u$ is three. Similarly, the degree of $v$ is three.
Note that $u$ can only become blue and only because of an active path of type $3\lra3$ ending at $u$.
This happens with probability $1-P_k(u,v,\gamma_u)$. Similarly, $v$ becomes blue with probability $1-P_k(v,u,\gamma_v)$.
Since the event that $u$ becomes blue and $v$ becomes blue conditioned by $c_{k-1}\left(T_{u,v}\right)=\gamma_u$ and
$c_{k-1}\left(T_{v,u}\right)=\gamma_v$ are independent, it follows that (\ref{eq-claim})
is also equal to $P_{k}(u,v,\gamma_u)\cdot P_{k}(v,u,\gamma_v)$ in this case.
\end{proof}

Lemma~\ref{claim} plays a crucial role in the Independence Lemma, which we now prove.
Its proof enlights how we designed our randomized procedure.

\begin{lemma}[Independence Lemma]
\label{indep}
Consider the randomized procedure for the infinite cubic tree.
Let $k$ be an integer, $uv$ an edge of the tree and
$\Gamma_u$ and $\Gamma_v$ two measurable subsets of $\C^\infty_{k-1}$.
Conditioned by the event $uv\subseteq W_{k}$,
the events that $c_{k}\left(T_{u,v}\right)\in\Gamma_u$ and $c_k\left(T_{v,u}\right)\in\Gamma_v$ are independent.
In other words,
$$\Pr\bigl[c_{k}\left(T_{u,v}\right)\in\Gamma_u \cond uv \subseteq~W_{k} \bigr]=
  \Pr\bigl[c_{k}\left(T_{u,v}\right)\in\Gamma_u \cond uv \subseteq~W_{k} \land c_k\left(T_{v,u}\right)\in\Gamma_v\bigr]\;\mbox{.}$$
\end{lemma}

\begin{proof}
The proof proceeds by induction on $k$.
If $k=1$, the event $uv\subseteq W_{1}$ implies that neither $u$, $v$ nor their neighbors
is active during the first round. Conditioned by this, the other vertices of the infinite
cubic tree marked active with probability $p_1$ randomly and independently. The result
of the marking in $T_{u,v}$ fully determine the coloring of the vertices of $T_{u,v}$ and
is independent of the marking and coloring of $T_{v,u}$. Hence, the claim follows.

Assume that $k>1$. Fix subsets $\Gamma_u$ and $\Gamma_v$.
We aim at showing that the probabilities
\begin{equation}
\label{eq-a}
\Pr\bigl[c_k\left(T_{u,v}\right)\in\Gamma_u \cond uv \subseteq W_k \bigr]
\end{equation}
and
\begin{equation}
\label{eq-b}
\Pr\bigl[c_k\left(T_{u,v}\right)\in\Gamma_u \cond u \in W_k \land c_k\left(T_{v,u}\right) \in \Gamma_v \bigr]
\end{equation}
are equal.

The definition of the conditional probability and the fact that $\Pr[uv\subseteq W_{k-1}|uv \subseteq W_k]=1$
yield that (\ref{eq-a}) is equal to
\begin{equation}
\frac{\Pr\bigl[c_k\left(T_{u,v}\right)\in\Gamma_u \land v \in W_k \cond uv \subseteq W_{k-1} \bigr]}
     {\Pr\bigl[uv \subseteq W_k \cond uv \subseteq W_{k-1} \bigr]}\;\mbox{.}
\label{eq-aa}
\end{equation}
By the induction,
for any two subsets $\Gamma_u'$ and $\Gamma_v'$ from $C^\infty_{k-1}$,
the probabilities
$\Pr\bigl[c_{k-1}\left(T_{v,u}\right)\in\Gamma_v' \cond u \in W_{k-1} \land c_{k-1}\left(T_{u,v}\right) \in \Gamma_u' \bigr]$
and
$\Pr\bigl[c_{k-1}\left(T_{v,u}\right)\in\Gamma_v' \cond uv \subseteq W_{k-1} \bigr]$
are the same. Hence, the numerator of (\ref{eq-aa}) is equal to the following:
\begin{align*}
\int_{\gamma_u', \gamma_v' \in C^\infty_{k-1}} &
\Pr\bigl[uv \subseteq W_k \cond c_{k-1}\left(T_{u,v}\right)=\gamma_u' \land c_{k-1}\left(T_{v,u}\right)=\gamma_v'\bigr]\times
\\*
&\Pr\bigl[c_k\left(T_{u,v}\right)\in\Gamma_u \cond uv \subseteq W_k \land c_{k-1}\left(T_{u,v}\right)=\gamma_u'  \land c_{k-1}\left(T_{v,u}\right)=\gamma_v' \bigr]
\\*
& \dint\Pr\bigl[c_{k-1}\left(T_{u,v}\right)=\gamma_u' \cond uv \subseteq W_{k-1} \bigr]
  \dint\Pr\bigl[c_{k-1}\left(T_{v,u}\right)=\gamma_v' \cond uv \subseteq W_{k-1} \bigr]
\mbox{.}
\end{align*}
Observe that when conditioning by $uv\subseteq W_k$,
the event $c_k\left(T_{u,v}\right)\in\Gamma_u$ is independent of $c_{k-1}\left(T_{v,u}\right)=\gamma_v'$.
Hence, the double sum can be rewritten to
\begin{align*}
\int_{\gamma_u', \gamma_v' \in C^\infty_{k-1}} &
\Pr\bigl[uv \subseteq W_k \cond c_{k-1}\left(T_{u,v}\right)=\gamma_u' \land c_{k-1}\left(T_{v,u}\right)=\gamma_v'\bigr] \times
\\*
& \Pr\bigl[c_k\left(T_{u,v}\right)\in\Gamma_u \cond uv \subseteq W_k \land c_{k-1}\left(T_{u,v}\right)=\gamma_u' \bigr]
\\*
& \dint\Pr\bigl[c_{k-1}\left(T_{u,v}\right)=\gamma_u' \cond uv \subseteq W_{k-1} \bigr]
  \dint\Pr\bigl[c_{k-1}\left(T_{v,u}\right)=\gamma_v' \cond uv \subseteq W_{k-1} \bigr]
\; \mbox{.}
\end{align*}
An application of Lemma~\ref{claim} then yields that the double sum is equal to
\begin{align*}
\int_{\gamma_u', \gamma_v' \in C^\infty_{k-1}} &
P_k(u,v,\gamma_u') \cdot P_k(v,u,\gamma_v') \cdot
\Pr\bigl[c_k\left(T_{u,v}\right)=\gamma_u \cond uv \subseteq W_k \land c_{k-1}\left(T_{u,v}\right)=\gamma_u' \bigr]
\\*
& \dint\Pr\bigl[c_{k-1}\left(T_{u,v}\right)=\gamma_u' \cond uv \subseteq W_{k-1} \bigr]
  \dint\Pr\bigl[c_{k-1}\left(T_{v,u}\right)=\gamma_v' \cond uv \subseteq W_{k-1} \bigr]
\; \mbox{.}
\end{align*}
Regrouping the terms containing $\gamma_u'$ only and $\gamma_v'$ only, we obtain that
the numerator of (\ref{eq-aa}) is equal to
\begin{align*}
 \biggl( \int_{\gamma_u' \in C^\infty_{k-1}} &
 P_k(u,v,\gamma_u) \times
 \Pr\bigl[c_k\left(T_{u,v}\right)\in\Gamma_u \cond uv \subseteq W_k \land c_{k-1}\left(T_{u,v}\right)=\gamma_u' \bigr]
\\*
&
 \dint\Pr\bigl[c_{k-1}\left(T_{u,v}\right)=\gamma_u' \cond uv \subseteq W_{k-1} \bigr]\biggr) \times
\\*
\biggl( \int_{\gamma_v' \in C^\infty_{k-1}} &
P_k(v,u,\gamma_v)
\dint\Pr\bigl[c_{k-1}\left(T_{v,u}\right)=\gamma_v' \cond uv \subseteq W_{k-1} \bigr]
 \biggr)
\; \mbox{.}
\end{align*}
Along the same lines, the denominator of~$(\ref{eq-aa})$ can be expressed as
\begin{align*}
 \biggl( \int_{\gamma_u' \in C^\infty_{k-1}} &
 P_k(u,v,\gamma_u) \dint\Pr\bigl[c_{k-1}\left(T_{u,v}\right)=\gamma_u' \cond uv \subseteq W_{k-1} \bigr]\biggr) \times
\\*
\biggl( \int_{\gamma_v' \in C^\infty_{k-1}} &
P_k(v,u,\gamma_v)
\dint\Pr\bigl[c_{k-1}\left(T_{v,u}\right)=\gamma_v' \cond uv \subseteq W_{k-1} \bigr]
 \biggr)
\; \mbox{.}
\end{align*}
Cancelling out the integral over $\gamma_v'$ which is the same in the numerator and the denominator of $(\ref{eq-aa})$,
we obtain that $(\ref{eq-a})$ is equal to
{\scriptsize
\begin{equation}
\frac{\int_{\gamma_u' \in C^\infty_{k-1}}
      P_k(u,v,\gamma_u')
      \Pr\bigl[c_k\left(T_{u,v}\right)\in\Gamma_u \cond uv \subseteq W_k \land c_{k-1}\left(T_{u,v}\right)=\gamma_u' \bigr]
      \dint\Pr\bigl[c_{k-1}\left(T_{u,v}\right)=\gamma_u' \cond uv \subseteq W_{k-1} \bigr]
      }
     {\int_{\gamma_u' \in C^\infty_{k-1}}
      P_k(u,v,\gamma_u')
      \dint\Pr\bigl[c_{k-1}\left(T_{u,v}\right)=\gamma_u' \cond uv \subseteq W_{k-1} \bigr]}
\;\mbox{.}\label{eq-a-final}
\end{equation}
}
The same trimming is applied to (\ref{eq-b}). First, the probability (\ref{eq-b})
is expressed as
\begin{equation}
\frac{\Pr\bigl[c_k\left(T_{u,v}\right)\in\Gamma_u \land c_k\left(T_{v,u}\right) \in \Gamma_v \cond uv \subseteq W_{k-1} \bigr]}
     {\Pr\bigl[u\in W_k \land c_k\left(T_{v,u}\right) \in \Gamma_v \cond uv \subseteq W_{k-1} \bigr]}
 \;\mbox{.}
\label{eq-bb}
\end{equation}
The numerator of $(\ref{eq-bb})$ is then expanded to
\begin{align*}
\biggl( \int_{\gamma_u' \in C^\infty_{k-1}} &
P_k(u,v,\gamma_u')
\Pr\bigl[c_k\left(T_{u,v}\right)\in\Gamma_u \cond uv \subseteq W_k \land c_{k-1}\left(T_{u,v}\right)=\gamma_u' \bigr]
\\*
&\dint\Pr\bigl[c_{k-1}\left(T_{u,v}\right)=\gamma_u' \cond uv \subseteq W_{k-1} \bigr] \biggr)
\times
\\*
\biggl( \int_{\gamma_v' \in C^\infty_{k-1}} &
\Pr\bigl[c_k\left(T_{v,u}\right)\in\Gamma_v \cond uv \subseteq W_k \land c_{k-1}\left(T_{v,u}\right)=\gamma_v'  \cdot P_k(v,u,\gamma_v')\bigr]
\\*
&\dint\Pr\bigl[c_{k-1}\left(T_{v,u}\right)=\gamma_v' \cond uv \subseteq W_{k-1} \bigr]  \biggr)
\end{align*}
and the denominator of $(\ref{eq-b})$ is expanded to
\begin{align*}
\biggl( \int_{\gamma_u' \in C^\infty_{k-1}} &
P_k(u,v,\gamma_u') \dint\Pr\bigl[c_{k-1}\left(T_{u,v}\right)=\gamma_u' \cond uv \subseteq W_{k-1} \bigr]\biggr)
\times \\*
\biggl( \int_{\gamma_v' \in C^\infty_{k-1}} &
\Pr\bigl[c_k\left(T_{v,u}\right)\in\Gamma_v \cond uv \subseteq W_k \land c_{k-1}\left(T_{v,u}\right)=\gamma_v'  \cdot P_k(v,u,\gamma_v')\bigr]
\\*
&\dint\Pr\bigl[c_{k-1}\left(T_{v,u}\right)=\gamma_v' \cond uv \subseteq W_{k-1} \bigr]  \biggr)
\; \mbox{.}
\end{align*}
We obtain (\ref{eq-a-final}) by cancelling out the integrals over $\gamma_v'$.
The proof is now finished.
\end{proof}

\subsection{Recurrence relations}
\label{sect-recurr}
We now derive recurence relations for the probabilities describing
the behavior of the randomized procedure.
We show how to compute the probabilities after $(k+1)$ rounds only from the probabilities after $k$ rounds.

Recall that $w^i_k$ is the probability that a fixed vertex $u$ has degree $i$ after $k$ rounds
conditioned by the event that $u$ is white after $k$ rounds.
Also recall that $q^i_k$ is the probability that a fixed vertex $u$ with a fixed neighbor $v$
has degree $i$ after $k$ rounds conditioned by the event that both $u$ and $v$ are white after $k$ rounds,
i.e., $uv \subseteq W_k$. Finally, $w_k, r_k$ and $b_k$ are probabilities that a fixed vertex is white, red and
blue, respectively, after $k$ rounds.

If $u$ is white,
the {\em white subtree} of $T_{u,v}$ is the maximal subtree containing $u$ and white vertices only.
We claim that the probability that the white subtree of $T_{u,v}$
is isomorphic to a tree in a given subset $\T_0$ after $k$ rounds,
conditioned by the event that both $u$ and $v$ are white after $k$ rounds,
can be computed from the values of $q^i_k$ only.
Indeed, if $T_0\in \T_0$,
the probability that $v$ has degree $i$ as in $T_0$ is $q^i_k$.
Now, if the degree of $v$ is $i$ as in $T_0$ and $z$ is a neighbor of $v$,
the values of $q^i_k$ again determine the probability that the degree of $z$ is as in $T_0$.
By Lemma~\ref{indep}, the probabilities that $v$ and $z$ have certain degrees,
conditioned by the event that they are both white, are independent. Inductively,
we can proceed with other vertices of $T_0$. Applying standard probability arguments,
we see that the values of $q^i_k$ fully determine the probability that the white subtree
of $T_{u,v}$ after $k$ rounds is isomorphic to a tree in $\T_0$.

After the first round, the probabilities $w_k, r_k$, $b_k$ and $q^i_k$, $i\in\{0,1,2,3\}$,
are the following.
\begin{align*}
& b_1 = 1 - \left(1-p_1\right)^3
& q^1_1 =& \left(1-\left(1-p_1\right)^2\right)^2
\\
& r_1 = p_1 \left(1-b_1\right) = p_1 \left(1-p_1\right)^3
& w^3_1 =& \left(1-p_1\right)^6
\\
& w_1 = 1 - b_1 - r_1
& w^2_1 =& 3 \cdot \left(1-p_1\right)^4 \left(1-\left(1-p_1\right)^2\right)
\\
& q^3_1 = \left(1-p_1\right)^4
& w^1_1 =& 3 \cdot \left(1-p_1\right)^2 \left(1-\left(1-p_1\right)^2\right)^2
\\
& q^2_1 = 2 \cdot \left(1-p_1\right)^2 \left(1-\left(1-p_1\right)^2\right)
& w^0_1 =& \left(1-\left(1-p_1\right)^2\right)^3
\end{align*}
A vertex becomes blue if at least one of its neighbors is active and
it becomes red if it is active and none of its neighbors is also active. Otherwise,
a vertex stays white. This leads to the formulas above.

To derive formulas for the probabilities $w_k, r_k$, $b_k$ and $q^i_k$ for $k\ge 2$,
we introduce additional notation.
The recurrence relations can be expressed using $q^i_k$ only,
but additional notation will help us to simplify expressions appearing in our analysis.
For a given edge $uv$ of the infinite tree,
let $P^{\to1}_k$ is the probability that the white subtree of $T_{u,v}$ contains a path
from $u$ to a vertex of degree one with all inner vertices of degree two after $k$ rounds
conditioned by the event $uv\subseteq W_k$. Note that such a path may end at $v$.
$P^{E\to1}_k$ and $P^{O\to1}_k$ are probabilities that the length of such a path is even or odd, respectively.
Analogously, $P^{\to3}_k$, $P^{E\to3}_k$ and $P^{O\to3}_k$ are probabilities that
the white subtree of $T_{u,v}$ contains a path, an even path and an odd path, respectively,
from $u$ to a vertex of degree three
with all inner vertices of degree two after $k$ rounds conditioned by the event $uv\subseteq W_k$.

Using Lemma~\ref{indep}, we conclude that the values of the just
defined probabilities can be computed as follows.
\begin{align*}
P^{\to1}_k =& q^1_k \cdot \textstyle\sum_{\ell\ge0}\left(q^2_k\right)^\ell  =  \frac{q^1_k}{1-q^2_k} &
&P^{\to3}_k = q^3_k \cdot \textstyle\sum_{\ell\ge0}\left(q^2_k\right)^\ell  =  \frac{q^3_k}{1-q^2_k}
\\
P^{O\to1}_k =& q^1_k \cdot \textstyle\sum_{\ell\ge0}\left(q^2_k\right)^{2\ell} =  \frac{q^1_k}{1-\left(q^2_k\right)^2} &
&P^{O\to3}_k = q^3_k \cdot \textstyle\sum_{\ell\ge0}\left(q^2_k\right)^{2\ell} =  \frac{q^3_k}{1-\left(q^2_k\right)^2}
\\
P^{E\to1}_k =& P^{\to1}_k - P^{O\to1}_k = q^2_k \cdot P^{O\to1}_k &
&P^{E\to3}_k = P^{\to3}_k - P^{O\to3}_k = q^2_k \cdot P^{O\to3}_k 
\end{align*}
Observe that $P^{\to1}_k + P^{\to3}_k = 1$.

The formulas for the above probabilities can be easily altered to express the probabilities that
a path exists and one of its inner vertices is active; simply, instead of multiplying
by $q^2_k$, we multiply by $p_2q^2_k$.
$\widehat{P}^{\to3}_k$ is now the probability
that the white subtree of $T_{uv}$ contains a path from $u$ to a vertex of degree three
with all inner vertices of degree two after $k$ rounds and none of them become active,
conditioned by $uv\subseteq W_k$.
Analogously to the previous paragraph,
we use $\widehat{P}^{O\to3}$ and $\widehat{P}^{E\to3}$.
The probabilities $\widehat{P}^{\to3}_k$, $\widehat{P}^{O\to3}$ and $\widehat{P}^{E\to3}$
can be computed in the following way.
\begin{align*}
\widehat{P}^{\to3}_k =& q^3_k \cdot \textstyle\sum_{\ell\ge0} \bigl(q^2_k \cdot\left(1-p_2\right)\bigr)^\ell =
\frac{q^3_k}{1-q^2_k\cdot\left(1-p_2\right)}
\\
\widehat{P}^{O\to3}_k =& q^3_k \cdot \textstyle\sum_{\ell\ge1}\bigl(q^2_k \cdot\left(1-p_2\right)\bigr)^{2\ell} =
\frac{q^3_k}{1-\left(q^2_k\right)^2\cdot\left(1-p_2\right)^2 } &
\\
\widehat{P}^{E\to3}_k =& \widehat{P}^{\to1}_k - \widehat{P}^{O\to3}_k = q^2_k\cdot\left(1-p_2\right) \cdot \widehat{P}^{O\to3}_k &
\end{align*}
The probabilities $\widetilde{P}^{O\to3}_k$ and $\widetilde{P}^{E\to3}$
are the probabilities that such an odd/even path exists and at least one of its inner vertices become active.
Note that $P^{O\to3}_k=\widehat{P}^{O\to3}_k+\widetilde{P}^{O\to3}_k$.
The value of $\widetilde{P}^{O\to3}_k$ is given by the equation
$$\widetilde{P}^{O\to3}_k = \left(q^2_k\right)^2 \cdot
\left(\left(1-p_2\right)^2 \cdot \widetilde{P}^{O\to3}_k + \left(1-\left(1-p_2\right)^2\right) \cdot P^{O\to3}_k\right)\;\mbox{,} $$
which can be manipulated to 
$$\widetilde{P}^{O\to3}_k = \frac{\left(q^2_k\right)^2 \cdot \left(1-\left(1-p_2\right)^2\right) \cdot P^{O\to3}_k}{1-\left(q^2_k\right)^2\cdot\left(1-p_2\right)^2} \; \mbox{.}$$
Using the expression for $\widetilde{P}^{O\to3}_k$,
we derive that $\widetilde{P}^{E\to3}_k$ is equal to the following.
$$ \widetilde{P}^{E\to3}_k = q^2_k\cdot \left( p_2 \cdot P^{O\to3}_k + \left(1-p_2\right) \cdot \widetilde{P}^{O\to3}_k \right) $$

We now show how to compute the probabilities $w_{k+1}, b_{k+1}$ and $r_{k+1}$.
Since blue and red vertices keep their colors once assigned, we have to focus
on the probability that a white vertex changes a color. We distinguish vertices
based on their degrees.
\begin{description}
\item[A vertex of degree zero.] Such a vertex is always recolored to red.
\item[A vertex of degree one.] Such a vertex is always recolored. Its new color is blue only
     if lies on an odd path to another vertex of degree one and the other end is chosen
     to be the beginning of the path. This leads to the following equalities.
     $$ \Pr\big[u \in R_{k+1} \cond u \in W^1_k \big] = \frac{1}{2}P^{O\to1}_k + P^{E\to1}_k + P^{\to3}_k \; \mbox{,} $$
     and
     $$ \Pr\big[u \in B_{k+1} \cond u \in W^1_k \big] = \frac{1}{2}P^{O\to1}_k \; \mbox{.} $$ 
\item[A vertex of degree two.] Since we have already computed the~probabilities that the paths
     of white vertices with degree two
     leading in the two directions from the vertex end at a vertex of degree one/three,
     have odd/even length and contain an active vertex, we can easily determine the probability
     that the vertex stay white or become red or blue. Note that for odd paths with type $1\lra1$ and
     $3\lra3$, in addition, the random choice of the start of the path comes in the play.
     It is then straightforward to derive the following.

{\scriptsize
\begin{align*}		
\Pr&\big[u \in R_{k+1} \cond u \in W^2_k \big] = \left(P^{E\to1}_k\right)^2 + \frac{2}{2} P^{O\to1}_k P^{E\to1}_k
+ 2 P^{E\to1}_k P^{\to3}_k
\\* & + (1-p_2) \cdot \left( \left(
	\widetilde{P}^{O\to3}_k\right)^2 
	+ 2\widetilde{P}^{O\to3}_k \widehat{P}^{O\to3}_k
	+ \frac{2}{2} \widetilde{P}^{O\to3}_k \widetilde{P}^{E\to3}_k
	+ \frac{2}{2} \widehat{P}^{O\to3}_k \widetilde{P}^{E\to3}_k
	+ \frac{2}{2} \widetilde{P}^{O\to3}_k \widehat{P}^{E\to3}_k
	\right)
\\* & + p_2 \cdot \left(\left(P^{O\to3}_k\right)^2 + \frac{2}{2} P^{E\to3}_k P^{O\to3}_k\right)
\\
\Pr&\big[u \in B_{k+1} \cond u \in W^2_k \big] = \left(P^{O\to1}_k\right)^2 + \frac{2}{2} P^{O\to1}_k P^{E\to1}_k
+ 2 P^{O\to1}_k P^{\to3}_k
\\* & + (1-p_2) \cdot \left( \left(
	\widetilde{P}^{E\to3}_k\right)^2
	+ 2\widetilde{P}^{E\to3}_k \widehat{P}^{E\to3}_k
	+ \frac{2}{2} \widetilde{P}^{O\to3}_k \widetilde{P}^{E\to3}_k
	+ \frac{2}{2} \widehat{P}^{O\to3}_k \widetilde{P}^{E\to3}_k
	+ \frac{2}{2} \widetilde{P}^{O\to3}_k \widehat{P}^{E\to3}_k
	\right)
\\* & + p_2 \cdot \left(\left(P^{E\to3}_k\right)^2 + \frac{2}{2} P^{E\to3}_k P^{O\to3}_k\right)
\end{align*}
}
\item[A vertex of degree three.] The vertex either stays white or is recolored to blue.
It stays white if (and only if) all the three paths with inner vertices being white and with degree two
are among the following: even paths to a vertex of degree one,
non-activated paths to a vertex of degree three, 
or activated odd paths to a vertex of degree three that was chosen as the beginning (and thus recolored with blue).
Hence we obtain that
\begin{equation}\label{eq-W3toB}
\Pr\big[u \in B_{k+1} \cond u \in W^3_k \big] = 1 - \left( P^{E\to1}_k + \widehat{P}^{\to3}_k + \frac{1}{2}\widetilde{P}^{O\to3}_k\right)^3
\; \mbox{.}
\end{equation}
\end{description}

Plugging all the probabilites from the above analysis together yields that
\begin{align*}
r_{k+1} =& r_k + w_k \cdot \left( \sum_{i=0}^2 w^i_k \cdot \Pr\big[u \in R_{k+1} \cond u \in W^i_k\big]  \right) 
\\
b_{k+1} =& b_k + w_k \cdot \left( \sum_{i=1}^3 w^i_k \cdot \Pr\big[u \in B_{k+1} \cond u \in W^i_k\big]  \right)
\\
w_{k+1} =& 1- r_{k+1} - b_{k+1}	\; \mbox{.}
\end{align*}

The crucial for the whole analysis is computing the values of $w^i_{k+1}$.
Suppose that $u$ is a white vertex after $k$ rounds.
The values of $w^i_k$ determine the probability that $u$ has degree $i$ and
the values of $q^i_k$ determine the probabilities that white neighbors of $u$ have certain degrees.
In particular, the probability that $u$ has degree $i$ and its neighbors have degrees
$j_1,\dots,j_i$ after $k$ rounds conditioned by $u$ being white after $k$ rounds is equal to
$$w^i_k \cdot \prod\limits_{j\in \left\{ j_1,\dots,j_i \right\}} q^j_k \; \mbox{.}$$
In what follows, the vector of degrees $j_1,\ldots,j_i$ will be denoted by $\vec{J}$.

Let $R^i_{k+1}(\vec{J})$ be the probability that $u$ is white after $(k+1)$ rounds
conditioned by the event that $u$ is white, has degree $i$ and
its white neighbors have degrees $\vec{J}$ after $k$ rounds.
Note that the value of $R^i_{k+1}(\vec{J})$ is the same for all permutation
of entries/degrees of the vector $\vec{J}$.

If $u$ has degree three and all its neighbors also have degree three after $k$ rounds,
the probability $R^3_{k+1}(3,3,3)$ is equal to one: no vertex of
degree three can be colored by red and thus the color of $u$ stays white.
On the other hand, if $u$ or any of its neighbor has degree one,
$u$ does definitely not stay white and the corresponding probability
$R^i_{k+1}(\vec{J})$ is equal to zero.

We now analyze the value of $R^i_{k+1}(\vec{J})$ for the remaining combinations of $i$ and $\vec{J}$.
If $i=2$, then $u$ stays white only if it lies on a non-active path of degree-two vertices
between two vertices of degree three. Consequently, it holds that
\begin{align*}
R^2_{k+1}(2,2) =& \left(1-p_2\right)^3 \cdot \left(\widehat{P}^{\to3}_k\right)^2
\\
R^2_{k+1}(3,2) =& \left(1-p_2\right)^2 \cdot \widehat{P}^{\to3}_k
\\
R^2_{k+1}(3,3) =& \left(1-p_2\right)
\end{align*}
If $i=3$, then $u$ stays white if and only if
for every neighbor $v$ of degree two of $u$,
the path of degree-two vertices from $v$ to $u$ is
\begin{itemize}
\item an odd path to a vertex of degree one, 
\item a non-activated path to a vertex of degree three,
\item an activated even path to vertex of~degree three and $u$ is not chosen as the beginning.
\end{itemize}
Based on this, we obtain that the values of $R^3_{k+1}(\vec{J})$ for the remaining
choices of $\vec{J}$ are the as follows.
\begin{align*}
R^3_{k+1}(2,2,2) =& \left( P^{O\to 1}_k + \left(1-p_2\right) \cdot \left(\widehat{P}^{\to 3}_k + \frac{1}{2}\widetilde{P}^{E\to 3}_k \right) + p_2\cdot \frac{1}{2} P^{E\to 3}_k \right)^3
\\
R^3_{k+1}(2,2,3) =& \left( P^{O\to 1}_k + \left(1-p_2\right) \cdot \left(\widehat{P}^{\to 3}_k + \frac{1}{2}\widetilde{P}^{E\to 3}_k \right) + p_2\cdot \frac{1}{2} P^{E\to 3}_k \right)^2
\\
R^3_{k+1}(2,3,3) =& P^{O\to 1}_k + \left(1-p_2\right) \cdot \left(\widehat{P}^{\to 3}_k + \frac{1}{2}\widetilde{P}^{E\to 3}_k \right) + p_2\cdot \frac{1}{2} P^{E\to 3}_k 
\\
\end{align*}

We now focus on computing the probabilities
$R^{i\to i'}_{k+1}(\vec{J})$ that a vertex $u$ is a white vertex of degree $i'$ after $(k+1)$ rounds
conditioned by the event that $u$ is a white vertex with degree $i$ with neighbors of degrees in $\vec{J}$
after $k$ rounds and $u$ is also white after $(k+1)$ rounds.
For example, $R^{2\to i'}_{k+1}(2,2)$ is equal to one for $i'=2$ and to zero for $i'\ne2$.
To derive formulas for the probabilites $R^{i\to i'}_{k+1}$, we have to introduce some additional notation.
$S^{(i,j)}_{k+1}$ for $(i,j)\in\{(2,2),(2,3),(3,2),(3,3)\}$
will denote the probability that a~vertex~$v$ is white after $(k+1)$ rounds
conditioned by the event that $v$ is a white vertex of degree $j$ after $k$ rounds and
a fixed (white) neighbor $u$ of $v$ has degree $i$ after $k$ rounds and $u$ is white after $(k+1)$ rounds.
It is easy to see that $S^{(2,2)}_{k+1} = 1$.
If $j=3$, the event we condition by guarantees that one of the neighbors of $v$ is white after $(k+1)$ rounds.
Hence, we derive that
$$S^{(2,3)}_{k+1} = S^{(3,3)}_{k+1} = \left( P^{E\to1}_k + \widehat{P}^{\to3}_k + \frac{1}{2}\widetilde{P}^{O\to3}_k\right)^2 \; \mbox{.}$$

Using the probabilities $S^{(i,j)}_{k+1}$, we can easily express some of the probabilities $R^{i\to i'}_{k+1}(\vec{J})$.

{\footnotesize
\begin{align*}
R^{2\to 2}_{k+1}(2,3) =& S^{(2,3)}_{k+1}
&
R^{2\to 2}_{k+1}(3,3) =& \left(S^{(2,3)}_{k+1}\right)^2
\\*
R^{2\to 1}_{k+1}(2,3) =& 1 - S^{(2,3)}_{k+1}
&
R^{2\to 1}_{k+1}(3,3) =& 2 \cdot S^{(2,3)}_{k+1} \left(1 - S^{(2,3)}_{k+1}\right)
\\*
R^{2\to 0}_{k+1}(2,3) =& 0
&
R^{2\to 0}_{k+1}(3,3) =& \left(1 - S^{(2,3)}_{k+1}\right)^2
\\
R^{3\to 3}_{k+1}(3,3,3) =& \left(S^{(3,3)}_{k+1}\right)^3
&
R^{3\to 1}_{k+1}(3,3,3) =& 3 \cdot S^{(3,3)}_{k+1} \left(1 - S^{(3,3)}_{k+1}\right)^2
\\*
R^{3\to 2}_{k+1}(3,3,3) =& 3 \cdot \left(S^{(3,3)}_{k+1}\right)^2 \left(1 - S^{(3,3)}_{k+1}\right)
&
R^{3\to 0}_{k+1}(3,3,3) =& \left(1 - S^{(3,3)}_{k+1}\right)^3
\end{align*}
}

We now determine $S^{(3,2)}_{k+1}$, i.e.,
the probability that a vertex $v$ is white after $(k+1)$ rounds conditioned by the event that
$v$ has degree two after $k$ rounds and
a fixed white neighbor $u$ of $u$ that has degree three after $k$ rounds is white after $(k+1)$ rounds.
Observe that $v$ is white after $(k+1)$ rounds
only if $v$ is contained in a non-active $3\lra3$ path.
Since we condition by the event that $u$ is white after $(k+1)$ rounds,
$v$ cannot be contained in an active $3\lra3$ path of even length or
an active $3\lra3$ odd path with $u$ being chosen as the beginning of this path.
Hence, the value of $S^{(3,2)}_{k+1}$ is the following.
$$S^{(3,2)}_{k+1} = \frac{\left(1-p_2\right) \cdot \widehat{P}^{\to 3}_k}
{P^{O\to 1}_k + \left(1-p_2\right) \cdot \left(\widehat{P}^{\to 3}_k + \frac{1}{2}\widetilde{P}^{E\to 3}_k \right) + p_2\cdot \frac{1}{2} P^{E\to 3}_k}
\; \mbox{.}$$
Using $S^{(3,2)}_{k+1}$, the remaining values of $R^{i\to i'}_{k+1}(\vec{J})$ can be expressed as follows.
{\footnotesize
\begin{align*}
R^{3\to3}_{k+1}(3,3,2) =& \left(S^{(3,3)}_{k+1}\right)^2 \cdot S^{(3,2)}_{k+1}
\\
R^{3\to2}_{k+1}(3,3,2) =& \left(S^{(3,3)}_{k+1}\right)^2 \left(1-S^{(3,2)}_{k+1}\right) + 2\cdot \left(1-S^{(3,3)}_{k+1}\right) S^{(3,3)}_{k+1} \cdot S^{(3,2)}_{k+1} 
\\
R^{3\to1}_{k+1}(3,3,2) =& \left(1-S^{(3,3)}_{k+1}\right)^2 S^{(3,2)}_{k+1} + 2\cdot S^{(3,3)}_{k+1} \left(1-S^{(3,3)}_{k+1}\right) \left(1-S^{(3,2)}_{k+1}\right)
\\
R^{3\to0}_{k+1}(3,3,2) =& \left(1-S^{(3,3)}_{k+1}\right)^2 \left(1-S^{(3,2)}_{k+1}\right)
\\
R^{3\to3}_{k+1}(2,2,3) =& \left(S^{(3,2)}_{k+1}\right)^2 \cdot S^{(3,3)}_{k+1}
\\
R^{3\to2}_{k+1}(2,2,3) =& \left(S^{(3,2)}_{k+1}\right)^2 \left(1-S^{(3,3)}_{k+1}\right) + 2\cdot \left(1-S^{(3,2)}_{k+1}\right) S^{(3,2)}_{k+1} \cdot S^{(3,3)}_{k+1} 
\\
R^{3\to1}_{k+1}(2,2,3) =& \left(1-S^{(3,2)}_{k+1}\right)^2 S^{(3,3)}_{k+1} + 2\cdot S^{(3,2)}_{k+1} \left(1-S^{(3,2)}_{k+1}\right) \left(1-S^{(3,3)}_{k+1}\right)
\\
R^{3\to0}_{k+1}(2,2,3) =& \left(1-S^{(3,2)}_{k+1}\right)^2 \left(1-S^{(3,3)}_{k+1}\right)
\\
R^{3\to 3}_{k+1}(2,2,2) =& \left(S^{(3,2)}_{k+1}\right)^3
\\
R^{3\to 2}_{k+1}(2,2,2) =& 3 \cdot \left(S^{(3,2)}_{k+1}\right)^2 \left(1 - S^{(3,2)}_{k+1}\right)
\\
R^{3\to 1}_{k+1}(2,2,2) =& 3 \cdot S^{(3,2)}_{k+1} \left(1 - S^{(3,2)}_{k+1}\right)^2
\\
R^{3\to 0}_{k+1}(2,2,2) =& \left(1 - S^{(3,2)}_{k+1}\right)^3
\end{align*}
}

Using $R^{i\to i'}_{k+1}(\vec{J})$,
we can compute $w^{i'}_{k+1}$, $i'\in \left\{0,1,2,3\right\}$.
In the formula below, $\J_i$ denotes the set of all possible vectors $\vec{J}$ with $i$ entries and
all entries either two or three.
The denominator of the formula is the probability that the vertex $u$ is white after $k+1$ rounds
conditioned by the event that $u$ is white after $k$ rounds; the nominator is the probability that
$u$ is white and has degree $i'$ after $k+1$ rounds
conditioned by the event that $u$ is white after $k$ rounds.

{\footnotesize
$$w^{i'}_{k+1}=\frac
{\displaystyle\sum_{\substack{i\ge2 \\ i\ge i'}} \sum\limits_{\vec{J}\in\J_i}
w^i_k \cdot \prod_{j\in\vec{J}} q^j_k \cdot R^i_{k+1}\left(\vec{J}\right) \cdot  R^{i\to i'}_{k+1}\left(\vec{J}\right)}
{\displaystyle\sum_{i\ge2} \sum\limits_{\vec{J}\in\J_i}w^i_k \cdot \prod_{j\in\vec{J}} q^j_k \cdot R^i_{k+1}\left(\vec{J}\right)}
$$
}

It remains to exhibit the recurrence relations for the values of $q^i_{k+1}$.
Let $uv$ be an edge of the tree, $i\ge 1$ and $\vec{J} \in \J_i$.
Observe that the probability that $u$ has degree $i$ and its neighbors have degrees $\vec{J}$
after $k$ rounds conditioned by the event $uv \subseteq W_k$ is exactly
$$q^i_k \cdot \prod\limits_{j \in \vec{J}} q^j_k \; \mbox{.}$$
In what follows, we will assume that the first coordinate $j_1$ of $\vec{J}$ corresponds
to the vector $v$.

The probabilites $Q^i_{k+1}$ are defined analogiclly to $R^i_{k+1}$ 
with an additional requirement that $v$ is also white after $k+1$ rounds.
Formally, $Q^i_{k+1}\left(\vec{J}\right)$ is the probability that a vertex $u$ and its fixed neighbor $v$
are both white after $(k+1)$ rounds conditioned by the event that $uv \subseteq W_k$,
$u$ has degree $i$ and its white neighbors have degrees $\vec{J}$ after $k$~rounds.
Observe that the following holds.
\begin{align*}
& Q^2_{k+1}(2,2) = R^2_{k+1}(2,2)
& Q^2_{k+1}(3,2) =& R^2_{k+1}(3,2) \cdot S^{(2,3)}_{k+1}
\\
& Q^2_{k+1}(2,3) = R^2_{k+1}(2,3) = R^2_{k+1}(3,2)
& Q^2_{k+1}(3,3) =& R^2_{k+1}(3,3) \cdot S^{(2,3)}_{k+1}
\\
& Q^3_{k+1}(j_1,j_2,j_3) = R^3_{k+1}(j_1,j_2,j_3) \cdot S^{(3,j_1)}_{k+1}
\end{align*}
Similarly, $Q^{i\to i'}_{k+1}$ is the probability that
a vertex $u$ is a white vertex of degree $i'\ge 1$ and its fixed neighbor $v$ is white after $(k+1)$~rounds
conditioned by the event that $u$ is a white vertex with degree $i$ with neighbors of degrees in $\vec{J}$
after $k$ rounds and $uv \subseteq W_{k+1}$. Using the arguments analogous to those to derive
the formulas for $R^{i\to i'}_{k+1}(\vec{J})$, we obtain the following formulas for $Q^{i\to i'}_{k+1}$.
We provide the list of recurrences to compute the values of $Q^{i\to i'}_{k+1}$ and
leave the actual derivation to the reader.

{\scriptsize
\begin{align*}
& Q^{2\to2}_{k+1}(2,2) =  Q^{2\to 2}_{k+1}(3,2) = 1
& Q^{2\to 2}_{k+1}(2,3) &=  Q^{2\to 2}_{k+1}(3,3) = S^{(2,3)}_{k+1}
\\*
& Q^{2\to1}_{k+1}(2,2) =  Q^{2\to 1}_{k+1}(3,2) = 0
& Q^{2\to 1}_{k+1}(2,3) &=  Q^{2\to 1}_{k+1}(3,3) = 1 - S^{(2,3)}_{k+1}
\end{align*}
\begin{align*}
& Q^{3\to 3}_{k+1}(3,3,3) = \left(S^{(3,3)}_{k+1}\right)^2
& Q^{3\to3}_{k+1}(3,2,2) =& \left(S^{(3,2)}_{k+1}\right)^2 
\\*
& Q^{3\to 2}_{k+1}(3,3,3) = 2 \cdot S^{(3,3)}_{k+1} \left(1 - S^{(3,3)}_{k+1}\right)
& Q^{3\to2}_{k+1}(3,2,2) =& 2\cdot \left(1-S^{(3,2)}_{k+1}\right) S^{(3,2)}_{k+1} 
\\*
& Q^{3\to 1}_{k+1}(3,3,3) = \left(1 - S^{(3,3)}_{k+1}\right)^2
& Q^{3\to1}_{k+1}(3,2,2) =& \left(1-S^{(3,2)}_{k+1}\right)^2
\\
& Q^{3\to3}_{k+1}(3,3,2) = S^{(3,3)}_{k+1} \cdot S^{(3,2)}_{k+1}
& Q^{3\to3}_{k+1}(2,3,3) =& \left(S^{(3,3)}_{k+1}\right)^2
\\*
& Q^{3\to2}_{k+1}(3,3,2) = S^{(3,3)}_{k+1} \left(1-S^{(3,2)}_{k+1}\right) + \left(1-S^{(3,3)}_{k+1}\right) S^{(3,2)}_{k+1} 
& Q^{3\to2}_{k+1}(2,3,3) =& 2\cdot \left(1-S^{(3,3)}_{k+1}\right) S^{(3,3)}_{k+1}
\\
& Q^{3\to1}_{k+1}(3,3,2) = \left(1-S^{(3,3)}_{k+1}\right) \left(1-S^{(3,2)}_{k+1}\right)
& Q^{3\to1}_{k+1}(2,3,3) =& \left(1-S^{(3,3)}_{k+1}\right)^2
\\
& Q^{3\to3}_{k+1}(2,2,3) = S^{(3,2)}_{k+1} \cdot S^{(3,3)}_{k+1}
& Q^{3\to 3}_{k+1}(2,2,2) =& \left(S^{(3,2)}_{k+1}\right)^2
\\
& Q^{3\to2}_{k+1}(2,2,3) = \left(1-S^{(3,2)}_{k+1}\right) S^{(3,3)}_{k+1} + S^{(3,2)}_{k+1} \left(1-S^{(3,3)}_{k+1}\right)
& Q^{3\to 2}_{k+1}(2,2,2) =& 2 \cdot \left(1 - S^{(3,2)}_{k+1}\right) S^{(3,2)}_{k+1} 
\\
& Q^{3\to1}_{k+1}(2,2,3) = \left(1-S^{(3,2)}_{k+1}\right) \left(1-S^{(3,3)}_{k+1}\right)
& Q^{3\to 1}_{k+1}(2,2,2) =& \left(1 - S^{(3,2)}_{k+1}\right)^2
\end{align*}
}

Using the values of $Q^{i}_{k+1}$ and $Q^{i\to i'}_{k+1}$,
we can compute the values of $q^{i'}_{k+1}$ for $i' \in \left\{1,2,3\right\}$. 
The denominator of the formula is the probability that the vertices $u$ and $v$ are white after $k+1$ rounds
conditioned by the event that $u$ and $v$ are white after $k$ rounds; the nominator is the probability that
$u$ are $v$ are white and $u$ has degree $i'$ after $k+1$ rounds
conditioned by the event that $u$ and $v$ are white after $k$ rounds.
{\footnotesize
$$q^{i'}_{k+1}=\frac
{\displaystyle\sum_{\substack{i\ge2 \\ i\ge i'}} \sum\limits_{\vec{J}\in\J_i}
q^i_k \cdot \prod_{j\in\vec{J}} q^j_k \cdot Q^i_{k+1}\left(\vec{J}\right) \cdot  Q^{i\to i'}_{k+1}\left(\vec{J}\right)}
{\displaystyle\sum_{i\ge2} \sum\limits_{\vec{J}\in\J_i}q^i_k \cdot \prod_{j\in\vec{J}} q^j_k \cdot Q^i_{k+1}\left(\vec{J}\right)}
$$
}

\subsection{Solving the recurrences}

The recurrences presented in this section were solved numerically using the Python program
provided in the Appendix. The particular choice of parameters used in the program
was $p_1=p_2=10^{-5}$ and $K=\rounds$. The choice of $K$ was made in such a way that
$w_K\le 10^{-6}$. We also estimated the precision of our calculations
based on the representation of float numbers to avoid rounding errors effecting
the presented bound on significant digits. Solving the recurrences for the above
choice of parameters we obtain that $r_K > \const$.

\section{High-girth graphs}
\label{sect-girth}

In this section, we show how to modify the randomized procedure from the previous section
that it can be applied to cubic graphs with large girth.
In order to use the analysis of the randomized procedure presented in the previous section,
we have to cope with the dependence of some of the events caused by the presence of cycles
in the graph. To do so, we introduce an additional parameter $L$ which will control the length
of paths causing the dependencies. Then, we will be able to guarantee that the probability that
a fixed vertex of a given cubic graph is red is at least $r_K-o(1)$ assuming the girth of the graph
is at least $8KL+2$.
In particular, if $L$ tends to infinity, we approach the same probability as for the infinite cubic tree.

\subsection{Randomized procedure}

We now describe how the randomized procedure is altered. Let $G$ be the given cubic graph.
We produce a sequence $G_0,\bar{G}_1,G_1,\ldots,\bar{G}_K, G_K$ of vertex-colored subcubic graphs; 
the only vertices in these graphs that have less than three neighbors will always be assigned a new color---black.
The graphs can also contain some additional vertices which do not correspond to the vertices of $G$;
such vertices will be called {\em virtual}.
The graph $G_0$ is the cubic graph $G$ with all vertices colored white.
Assume that $G_{k-1}$ is already defined. Let $\bar{G}_{k}$ be the graph obtained from $G_{k-1}$
using the randomized procedure for the $k$-th round exactly as described for the infinite tree.

Before we describe how the graph $G_k$ is obtained from $\bar{G}_k$,
we need to introduce additional notation.
Let $y_0,y_1,\dots,y_{2L}$ be a fixed path in the cubic tree.
Now consider all possible colorings of $T_{y_1,y_0}$ after $k$ rounds 
that satisfy that
\begin{enumerate}
\item the vertices $y_0,y_1,\dots,y_{2L}$ are white after $k$ rounds and
\item the vertices $y_1,\dots,y_{2L-1}$ have degree two.
\end{enumerate}
Let $\D_k$ be the probability distribution on the colorings of $T_{y_1,y_0}$ that satisfy these two constraints
such that the probability of each coloring is proportional to its probability after $k$ rounds. In other words,
we discard the colorings that do not satisfy the two constraints and normalize the probabilities.

The graph $G_k$ is obtained
from $\bar{G}_k$ by performing the following operation for every path $P$ of type $1\lra1$, $1\lra3$ or $3\lra3$
between vertices $a$ and $b$ that has length at least $2L$ and contains at least one non-virtual inner vertex.
Let $x_u$ be the non-virtual inner vertex of $P$ that is the closest to $a$ and $x_v$ the one closest to $b$.
Let $P_x$ be the subpath between $x_u$ and $x_v$ (inclusively) in $\bar{G}_{k}$.
The process we describe will guarantee that
the non-virtual vertices of $P$ form a subpath of $P$,
i.e., $P_x$ containts exactly non-virtual inner vertices of $P$.
Let $u$ be the neighbor of $x_u$ on $P$ towards $a$, and $v$ the neighbor of $x_v$ on $P$ towards $b$;
$u$ and $a$ are the same if $a$ is non-virtual. Similarly, $v$ and $b$ are the same
if $b$ is non-virtual.

We now modify the graph $\bar{G}_k$ as follows.
Color the vertices of $P_x$ black and remove the edges $x_uu$ and $x_vv$ from the graph.
Then attach to $u$ and $v$ rooted trees $T_u$ and $T_v$,
all of them fully comprised of virtual vertices, such that the colorings
of $T_u$ and $T_v$ are randomly sampled
according to the distribution $\D_{k}$.
The roots of the trees will become adjacent to $u$ or $v$, respectively.
These trees are later referred to as {\em virtual trees}.
Observe that we have created no path between two non-virtual vertices containing a virtual vertex.

After $K$ rounds, the vertices of $G$ receive colors of their counterparts in $G_K$.
In this way, the vertices of $G$ are colored white, blue, red and black and
the red vertices form an independent set.

\subsection{Refining the analysis}

We argue that the analysis for the infinite cubic trees presented in Section~\ref{sect-tree}
also applies to cubic graphs with large girth.
We start with some additional definitions.
Suppose that $G_k$ is the graph obtained after $k$ rounds of the randomized
procedure. 
Let $u$ and $v$ be two vertices of $G_k$.
The vertex $v$ is {\em reachable} from $u$
if both $u$ and $v$ are white and
there exists a path in $G_k$ between $u$ and $v$ comprised of white vertices
with all inner vertices having degree two (recall that the degree of a vertex
is the number of its white neighbors).
Clearly, the relation of being reachable is symmetric.
The vertex $v$ is {\em near} to $u$
if both $u$ and $v$ are white and
either $v$ is reachable from $u$ or $v$ is a neighbor of a white vertex reachable from $u$.
Note that the relation of being near is not symmetric in general.
For a subset $X\subseteq V\left(G_k\right)$ of white vertices,
$N_k(X) \subseteq V\left(G_k\right)$ is the set of white vertices that are near to a vertex of $X$ in $G_k$.

We now prove the following theorem.

\begin{theorem}
Let $K$ be an integer, $p_1$ and $p_2$ positive reals, 
$G$ a cubic graph and $v$ a vertex of $G$.
For every $\varepsilon >0$ 
there exists an integer $L$ such that 
if $G$ has girth at least $8KL+2$, then 
the probability that the vertex $v$ will be red in $G$ at the end of the randomized procedure
is at least $r_K-\varepsilon$
where $r_K$ is the probability that a fixed vertex of the infinite cubic tree
is red after $K$ rounds of the randomized procedure with parameters $p_1$ and $p_2$.
\end{theorem}

\begin{proof}
We keep the notation introduced in the description of the randomized procedure.
As the first step in the proof, we establish the following two claims.

\begin{claim}
\label{cl-1}
Let $k$ be a non-negative integer,
$u$ a vertex of $G_k$, $c$ one of the colors,
$\gamma_1$ and $\gamma_2$ two colorings of vertices of $G_k$ such that
$u$ is white in both $\gamma_1$ and $\gamma_2$.
If the set of vertices near to $u$ in $\gamma_1$ and $\gamma_2$ induce isomorphic trees rooted at $u$,
then the probability that $u$ has the color $c$ in $\bar{G}_{k+1}$
conditioned by the event that $G_k$ is colored as in $\gamma_1$, and
the probability that $u$ has the color $c$ in $\bar{G}_{k+1}$
conditioned by the event that $G_k$ is colored as in $\gamma_2$, are the same.
\end{claim}

Indeed, the color of $u$ in $\bar{G}_{k+1}$ is influenced only by the length and the types
of the white paths in $G_k$ containing $u$. All the vertices of these paths are near to $u$ 
as well as the white neighbors of the other end-vertices of these paths (which are necessary
to determine the types of the paths).
By the assumption on the colorings $\gamma_1$ and $\gamma_2$,
the two probabilities from the statement of the claim are the same.

We now introduce yet another definiton.
For a subcubic graph $H$ with vertices colored white, red, blue or black,
a vertex $v \in H$ is said to be {\em $d$-close} to a white vertex $u \in H$ if
$v$ is white and there exists a path $P$ from $u$ to $v$ comprised of white vertices such that
\begin{itemize}
\item the length of $P$ is at most $d$, or
\item $P$ contains a vertex $w$ such that $w$ is at distance at most $d$ from $u$ on $P$ and
      each of the first $(2L-1)$ vertices following $w$ (if they exist) has degree two (recall that
      the degree of a vertex refers to the number of its white neighbors).
\end{itemize}
Finally,
a $d$-close tree of $u$ is the subgraph comprised of all vertices $d$-close to $u$ (for our choice
of $d$, this subgraph will always be tree) that is rooted at $u$. By the definition,
the $d$-close tree of $u$ contains white vertices only.

Observe that if $v$ is $d$-close to $u$, then it is also $d'$-close to $u$ for every $d'>d$.
Also observe that if a vertex $v$ of a virtual subtree is $d$-close to a non-virtual vertex $u$,
then all the white vertices lying in the same white component of the virtual subtree
are also $d$-close to $u$:
indeed, consider a path $P$ witnessing that $v$ is $d$-close to $u$ and
let $P_0$ be its subpath from $u$ to the root of the virtual tree.
The path $P_0$ witnesses that the root is $d$-close to $u$ and
$P_0$ can be prolonged by the path comprised of $2L-1$ degree-two virtual
vertices to a path to any white vertex $v'$ of the same white component as $v$.
The new path now witnesses that $v'$ is also $d$-close to $u$.

Let us look at $d$-close sets in the infinite cubic trees.

\begin{claim}
\label{cl-3}
Let $d$ be a non-negative integer,
$T$ an infinite cubic tree with vertices colored red, blue and white, and
$u$ a white vertex of $T$.
If a vertex $v$ is $d$-close to $u$ in $T$,
then every vertex $v'$ that is near to $v$ is $(d+2L)$-close to $u$.
\end{claim}

Let $P$ be a path from $u$ to $v$ that witnesses that $v$ is $d$-close to $u$.
Assume first that the length of $P$ is at most $d$. If $v'$ lies on $P$ or
is a neighbor of a vertex of $P$, then $v'$ is $(d+1)$-close to $u$.
Otherwise, consider the path $P'$ from $u$ to $v'$; observe that
$P$ is a subpath of $P'$ and all the vertices following $v$ on $P$
with a possible exception of $v'$ and the vertex immediately preceding it
have degree two. If the length of $P'$ is at most $d+2L$, then $v'$
is $(d+2L)$-close to $u$. Otherwise, $v$ is followed by at least $2L-1$
vertices of degree two and $v'$ is again $(d+2L)$-close to $u$.

Assume now that the length of $P$ is larger than $d$. Then, $P$ contains
a vertex $w$ at distance at most $d$ from $u$ such that the first $2L-1$ vertices
following $w$ (if they exist) have degree two.
If $v'$ lies on $P$ or is adjacent to a vertex of $P$,
then $v'$ lies on $P$ after $w$ or is adjacent to $w$. In both cases,
$v'$ is $(d+1)$-close to $u$.
In the remaining case, we again consider the path $P'$ from $u$ to $v'$
which must be an extension of $P$ (otherwise, $v'$ would lie on $P$ or
it would be adjacent to a vertex on $P$).
If there are at least $2L-1$ vertices following $w$ on $P$,
then $v'$ is $d$-close to $u$.
Otherwise, either $P'$ contains $2L-1$ vertices of degree two following $w$ or
the length of $P'$ is at most $d+2L$. In both cases, $v'$ is $(d+2L)$-close to $u$.

\bigskip

We are ready to prove our main claim.

\begin{claim}
\label{cl-5}
Let $k\le K-1$ be a positive integer, $T$ a rooted subcubic tree such that
its root is not contained in a path of degree-two vertices of length at least $2L$,
$u$ a vertex of the infinite cubic tree and
$v$ a non-virtual vertex of $G_k$.
The probability that $u$ is white and the $4(K-k)L$-close tree of $u$ in the infinite tree
is isomorphic to $T$ after $k$ rounds is the same as
the probability that $v$ is white and the $4(K-k)L$-close tree of $v$ in $G_k$
is isomorphic to $T$ after $k$ rounds.
\end{claim}

The proof proceeds by induction on $k$.
For $k=0$, both in the infinite tree and in $G_0$,
the probability is equal to one if $T$ is the full rooted cubic tree of depth $4KL$ and
it is zero, otherwise. Here, we use the girth assumption to derive that the subgraph of $G$ induced
by $4KL$-close vertices to $u$ is a tree (otherwise, $G$ would contain a cycle of length at most $8KL+1$).

Suppose $k>0$. Let $\widetilde{T}$ be the $4(K-k)L$-close tree of $v$ in $G_k$ and
$W$ the set of vertices that are $4(K-k)L+2L$-close to $v$ in $G_{k-1}$.
The degree and the color of each vertex in $\bar{G}_k$ is determined by vertices that are near to it in $G_{k-1}$
by Claim~\ref{cl-1}. The induction assumption and Claim~\ref{cl-3} imply that
every vertex of $W$ is white and has a given degree $i$ in $\bar{G}_k$
with the same probability as its counterpart in the infinite cubic tree
assuming that the vertex $u$ of the infinite tree does not lie
on a path of degree-two vertices of length at least $2L$ after $k-1$ rounds.

If $v$ lies on a path with at least $2L-1$ degree-two vertices in $\bar{G}_k$, it becomes black.
Othwerwise, the set $W$ contains all vertices that are $4(K-k)L$-close to $v$ with the exception of
the new virtual vertices that are $4(K-k)L$-close to $v$. Since the colorings of newly added virtual trees
have been sampled according to the distribution $\D_{k+1}$, Lemma~\ref{indep} implies that
the probability that $v$ is white and the $4(K-k)L$-close tree of $v$ is equal to $T$ is the same as
the corresponding probability for $u$ in the infinite cubic tree.

\bigskip

\begin{claim}
\label{cl-6}
Let $k$ be a non-negative integer and $u$ a vertex of the infinite cubic tree.
The probability that $u$ is white and
lies on a white path of length at least $2L$ after $k$ rounds
is at most $2 \cdot \left(q^2_{k}\right)^{L-1}$.
\end{claim}

If $u$ lies on such a path, its degree must be two and 
the length of the path from $u$ in one of the two directions
is at least $L-1$.
The probability that this happens for each of the two possible directions from $u$
is at most $\left(q^2_{k}\right)^{L-1}$. The claim now follows.

\bigskip

Let $p_0=\sum_{k=1}^K 2\left(q^2_{k}\right)^{L-1}$. Observe that $q^2_k<1$.
Indeed, if a vertex $u$ of the infinite tree and all the vertices at distance at most $2K$
from $u$ are white after the first round, then $u$ and its three neighbors must have degree three after $K$ rounds (all
vertices at distance at most $2(K+1-k)$ from $u$ are white after $k$ rounds). Since this happens with non-zero
probability, $q^3_k>0$ and thus $q^2_k<1$. This implies that $p_0$ tends to $0$ with $L$ approaching
the infinity.

Fix a vertex $v$ of $G$.
By Claim~\ref{cl-1}, the probability that $v$ is colored red in the $k$-th round
is fully determined by the vertices that are near to $v$ in $G_{k-1}$.
All such vertices are also $2L$-close to $v$ by Claim~\ref{cl-3}.
Consider a rooted subcubic tree $T$.
If the root of $T$ does not lie on a path with inner vertices
of degree two with length at least $2L$, then the probability that
$v$ is white and the $2L$-close tree of $v$ is isomorphic to $T$
is the same as the analogous probability for a vertex of the infinite tree by Claim~\ref{cl-5}.
Since the probability that $v$ is white and it lies on a path of length at least $2L$ in its $2L$-close tree
at some point during the randomized procedure is at most $p_0$ by Claim~\ref{cl-6},
the probability that $v$ is colored red in $G_K$ is at least $r_K-p_0$.
Since $p_0<\varepsilon$ for $L$ sufficiently large, the statement of the theorem follows.
\end{proof}

\section{Conclusion}

The method we presented here, similarly to the method of Hoppen~\cite{bib-hoppen},
can be applied to $r$-regular graphs for $r\ge 4$.

Another related question is whether
the fractional chromatic number of cubic graphs is bounded away from $3$ under a weaker
assumption that the odd girth (the length of the shortest odd cycle) is large.
This is indeed the case as we now show.

\begin{theorem}
\label{thm-odd}
Let $g\ge 5$ be an odd integer.
The fractional chromatic number of every subcubic graph $G$ with odd girth at least $g$
is at most $\frac{8}{3-6/(g+1)}$.
\end{theorem}

\begin{proof}
Clearly, we can assume that $G$ is bridgeless. If $G$ contains two or more vertices
of degree two, then we include $G$ in a large cubic bridgeless graph with the same
odd girth. Hence, we can assume that $G$ contains at most one vertex of degree two.
Consequently, $G$ has a $2$-factor $F$.

We now construct a probability distribution on the independent sets such that
each vertex is included in the independent set chosen according to this distribution
with probability at least $3(1-2/(g+1))/8$. This implies the claim of the theorem.

Number the vertices of each cycle of $F$ from $1$ to $\ell$ where $\ell$ is the length of the cycle.
Choose randomly a number $k$ between $1$ and $(g+1)/2$ and let $W$ be the set of all vertices
with indices equal to $k$ modulo $(g+1)/2$. Hence, each vertex is not in $W$ with probability
$1-2/(g+1)$.

Let $V_1,\ldots,V_m$ be the sets formed by the paths of $F\setminus W$. Since each set $V_i$,
$i=1,\ldots,m$, contains at most $g-1$ vertices, the subgraph $G[V_i]$ induced in $G$ by $V_i$
is bipartite. Choose randomly (and independently of the other subgraphs)
one of its two color classes and color its vertices red.
Observe that if an edge has both its end-points colored red,
then it must an edge of the matching $M$ complementary to $F$.
If this happens, choose randomly one vertex of this edge and uncolor it.

The resulting set of red vertices is independent. We estimate the probability that a vertex $v$ is red
conditioned by $v\not\in W$. With probability $1/2$, $v$ is initially colored red. However,
with probability at most $1/2$ its neighbor through an edge of $M$ is also colored red (it can happen
that this neighbor is in $W$). If this is the case, then the vertex $v$ is uncolored with probability $1/2$.
Consequently, the probability that $v$ is red is at least $1/2\cdot (1-1/4)=3/8$. Multiplying
by the probability that $v\not\in W$, which is $1-2/(g+1)$, we obtain that the vertex $v$ is included
in the independent set with probability at least $3(1-2/(g+1))/8$ as claimed earlier.
\end{proof}

\newpage

\section*{Appendix}

\begin{python}[label=prg-indep]
W_THOLD = 1.0/1000000
p_1 = 0.00001
p_2 = 0.00001

def state():
   print("%d: r=%f b=%f w=%f q3=%f q2=%f q1=%f w3=%f w2=%f w1=%f w0=%f" 
       % (k, p_r, p_b, p_w, q[3], q[2], q[1], w[3], w[2], w[1], w[0]))

def gen_C():
   S= {1: q[1], 2: q[2], 3: q[3]}
   for u in (1,2,3):
      A=(1,2,3)
      B=S if (u>1) else {"-":1}
      C=S if (u>2) else {"-":1}
      for a in A:
       for b in B:
        for c in C:
           C3[(u,a,b,c)] = w[u]*q[a]*B[b]*C[c]
           C2[(u,a,b,c)] = q[u]*q[a]*B[b]*C[c]


def init():
   global p_r, p_w, p_b

   t = (1-p_1)**2
   w[3] = t**3
   w[2] = 3*(t**2)*(1-t)
   w[1] = 3*t*(1-t)**2
   w[0] = (1-t)**3

   q[3] = t**2
   q[2] = 2*t*(1-t)
   q[1] = (1-t)**2

   p_b = 1-(1-p_1)**3
   p_r = p_1*(1-p_b)
   p_w = 1 - p_r - p_b

k=0
p_w = 1
p_b = p_r = 0

w = [0,0,0,3]
q = [0,0,0,1]
C3 = {}
C2 = {}

while (p_w > W_THOLD):
   k+=1

   if (k==1):
      init()
      state()
      continue
   o1 = q[1]/(1-q[2]**2) 
   e1 = q[2]*o1
   p1 = o1 + e1
   
   o3 = q[3]/(1-q[2]**2) 
   e3 = q[2]*o3
   p3 = o3 + e3

   p3_n = q[3]/(1-q[2]*(1-p_2))
   o3_n = q[3]/(1-q[2]**2 * (1-p_2)**2)
   e3_n = q[2]*(1-p_2)*o3_n

   o3_y = (q[2]**2 *(1-(1-p_2)**2)*o3)/(1-q[2]**2 *(1-p_2)**2)
   e3_y = q[2]*(p_2*o3 + (1-p_2)*o3_y)

   Pr = [1,0,0,0]
   Pb = [0,0,0,0]
   
   Pr[1] = (e1 + o1*.5) + p3
   Pb[1] = o1*.5

   Pr[2] = e1**2 + e1*o1 + 2*e1*p3
   Pr[2]+= (1-p_2)*(o3_y**2 +2*o3_y*o3_n+o3_y*e3_y+o3_n*e3_y+o3_y*e3_n)
   Pr[2]+= p_2*(o3**2 + o3*e3)

   Pb[2] =  o1**2 + e1*o1 + 2*o1*p3
   Pb[2]+= (1-p_2)*(e3_y**2 +2*e3_y*e3_n+o3_y*e3_y+o3_n*e3_y+o3_y*e3_n)
   Pb[2]+= p_2*(e3**2+o3*e3)

   Pb[3] = 1-(p3_n+e1+o3_y/2)**3

   p_r+= p_w*(w[0]+w[1]*Pr[1]+w[2]*Pr[2])
   p_b+= p_w*(w[1]*Pb[1]+w[2]*Pb[2]+w[3]*Pb[3])
   p_w = 1-p_r-p_b
   
   r32 = o1+(1-p_2)*(p3_n+.5*e3_y)+p_2*(.5*e3) 
   s33 = (p3_n+e1+o3_y/2)**2
   s32 = (1-p_2)*p3_n / r32

   gen_C()

   T = [0,0,0,0]
   sum = 0
   for C in C3:
      deg1=False
      for x in C:
         if x==1: 
            deg1=True
            break
      if deg1: continue

      N={}
      C_p = C3[C]




      if (C[0]==2):
          N[3] = {0:1}
         C_p*=1-p_2
         for i in (1,2):
            if (C[i]==3):
               N[i] = {0: s33 }
               N[i][1] = 1-N[i][0]
            elif (C[i]==2):
               C_p *= (1-p_2)*p3_n
               N[i] = {0: 1}

      elif (C[0]==3):
         for i in (1,2,3):
            if (C[i]==3):
               N[i] = {0: s33 }
            elif (C[i]==2):
               C_p *= r32
               N[i] = {0: s32 }

            N[i][1] = 1-N[i][0]

       sum += C_p
      for a in N[1]:
         for b in N[2]:
            for c in N[3]:
               T[C[0]-a-b-c] += C_p*N[1][a]*N[2][b]*N[3][c]

   for i in (0,1,2,3): w[i] = T[i]/sum 
         
   T = [0,0,0,0]
   sum = 0
   for C in C2:
      deg1=False
      for x in C:
         if x==1: 
            deg1=True
            break
      if deg1: continue

      N={}
      C_p = C2[C]

      if (C[1]==3):
         C_p *= s33
      elif (C[1]==2):
         C_p *= (1-p_2)*p3_n

      if (C[0]==2):
          N[3] = {0:1}
         C_p*=1-p_2

         if (C[2]==3):
            N[2] = {0: s33 }
            N[2][1] = 1-N[2][0]
         elif (C[2]==2):
            C_p *= (1-p_2)*p3_n
            N[2] = {0: 1}

      elif (C[0]==3):
         for i in (2,3):
            if (C[i]==3):
               N[i] = {0: s33 }
            elif (C[i]==2):
               C_p *= r32
               N[i] = {0: s32 }

            N[i][1] = 1-N[i][0]
 
       sum += C_p
      for b in N[2]:
         for c in N[3]:
            T[C[0]-b-c] += C_p*N[2][b]*N[3][c]
            
   for i in (1,2,3): q[i] = T[i]/sum 

   state()
\end{python}
\end{document}